\documentclass[10pt,a4paper]{elsarticle}
\usepackage{amsmath,amssymb,latexsym,amscd,amsthm}
\usepackage[T1]{fontenc}
\usepackage[english]{babel}
\usepackage{graphicx}

\newcommand{\tr}{\text{tr}}
\newcommand{\one}{\mbox{$1 \hspace{-1.0mm}  {\bf l}$}}

\newtheorem{theorem}{Theorem}
\newtheorem{remark}[theorem]{Remark}
\newtheorem{corollary}[theorem]{Corollary}
\newtheorem{lemma}[theorem]{Lemma}
\newtheorem{definition}[theorem]{Definition}

\begin{document}

\title{Characterization of preorders induced by positive maps in the set of Hermitian matrices}

\date{}

\author{Julio I. de Vicente}\ead{jdvicent@math.uc3m.es}
\address{Departamento de Matemáticas, Universidad Carlos III de Madrid,\\ Avda. de la Universidad 30, 28911, Leganés (Madrid), Spain}

\begin{abstract}
Uhlmann showed that there exists a positive, unital and trace-preserving map transforming a Hermitian matrix $A$ into another $B$ if and only if the vector of eigenvalues of $A$ majorizes that of $B$. In this work I characterize the existence of such a transformation when one of the conditions of unitality or trace preservation is dropped. This induces two possible preorders in the set of Hermitian matrices and I argue how this can be used to construct measures quantifying the lack of positive semidefiniteness of any given Hermitian matrix with relevant monotonicity properties. It turns out that the measures in each of the two formalisms are essentially unique.
\end{abstract}


\maketitle

\section{Introduction, main results and conclusions}

Let $M_n$ denote the set of square matrices with complex entries of size $n$ and $H_n$ the subset of Hermitian matrices in $M_n$. A positive semidefinite (negative semidefinite) matrix $A\in H_n$ is denoted by $A\geq0$ ($A\leq0$). A positive map is a linear map $\Phi:H_n\to H_k$ such that $\Phi(A)\geq0$ $\forall A\geq0$. A linear map $\Phi: M_n\to M_k$ is said to be trace-preserving if $\tr\Phi(A)=\tr A$ $\forall A\in M_n$, where $\tr$ stands for the trace of a matrix, and it is said to be unital if $\Phi(\one)=\one$, where $\one$ stands for the identity matrix (in order to ease the notation I do not specify the dimension of the identity matrix when it should be clear from the context). The following theorem is due to Uhlmann \cite{Uhl1,Uhl2,Uhl3} (see also \cite{Ando}):

\begin{theorem}
Let $A,B\in H_n$ with respective eigenvalues arranged in non-increasing order $\{\lambda_i(A)\}_{i=1}^n$ and $\{\lambda_i(B)\}_{i=1}^n$. Then, there exists a positive, unital and trace-preserving linear map $\Phi:H_n\to H_n$ such that $\Phi(A)=B$ if and only if the vector of eigenvalues of $A$ majorizes that of $B$, i.e.
\begin{align}
\sum_{i=1}^p\lambda_i(A)&\geq \sum_{i=1}^p\lambda_i(B),\quad 1\leq p\leq n-1\nonumber\\
\sum_{i=1}^n\lambda_i(A)&= \sum_{i=1}^n\lambda_i(B).
\end{align}
%where, if $n<k$ it should be understood that $\lambda_{n+1}(A)=\lambda_{n+2}(A)=\cdots=\lambda_k(A)=0$, and similarly for %$\{\lambda_i(B)\}_{i=k+1}^n$ if $n>k$.
\end{theorem}

This theorem extends to matrices the well-known result that there exists a $n\times n$ doubly stochastic matrix with non-negative entries $D$ such that $Dx=y$ for some $x,y\in\mathbb{R}^n$ if and only if $x$ majorizes $y$ \cite{Marshall}. This is because the condition of trace preservation is analogous to column stochasticity (i.e.\ $\sum_i(Dx)_i=\sum_ix_i$) and unitality to row stochasticity (i.e.\ $D{\bf 1}={\bf 1}$, where ${\bf 1}$ is the vector of all ones).

In this work I ask what the situation is if one of the two conditions that the positive map is unital or trace-preserving is dropped and I characterize for which matrices $A\in H_n$ and $B\in H_k$ there exists a positive unital (PU) linear map $\Phi:H_n\to H_k$ such that $\Phi(A)=B$ and the analogous question for positive trace-preserving (PTP) linear maps (notice that the condition that the map is both unital and trace-preserving imposes that $\Phi:H_n\to H_n$ while this is no longer the case if one of the conditions is lifted). Besides being a natural relaxation of the problem at hand that turns out to have a very compact answer, this is also motivated by the idea of developing measures that quantify the lack of positive semidefiniteness (or the lack of negative semidefiniteness) for Hermitian matrices. The standard approach in matrix theory to quantify the lack of structure of a matrix is by its distance in some norm to the set of structured matrices. Here I take a radically different starting point where the meaningfulness of the measure comes from fulfilling a monotonicity property with respect to a preorder. This is inspired by quantum resource theories that are used in quantum information theory in order to provide means to quantify to what degree a quantum state has a certain property based on the fact that a certain subset of quantum states does not have this property at all \cite{Chitambar}. In our case, this latter subset is the set of positive semidefinite matrices and we want to quantify how non-positive-semidefinite a Hermitian matrix might be.

Notice that the set of positive maps contains the identity map and is closed under composition; thus, this immediately implies that the existence of a positive map $\Phi$ such that $\Phi(A)=B$ induces a preorder in the set of Hermitian matrices, $A\to B$. Notice in addition that positive semidefinite matrices are ``at the bottom'' of this ordering relation: if $A$ is neither positive nor negative semidefinite, it then holds that $A\to B$ $\forall B\geq0$ but not the other way around (that $A\to B$ follows by considering e.g.\ the positive map $\Phi(X)=\frac{x^\ast Xx}{x^\ast Ax}B$ by choosing a vector $x$ such that $x^\ast A x>0$; that $B\nrightarrow A$ follows from the very definition of positive maps). Therefore, it seems natural to conclude that if $A\to B$, then $A$ is at least as non-positive-semidefinite as $B$. The same applies to negative semidefinite matrices. Due to the linearity of the maps, i.e. $\Phi(-A)=-\Phi(A)$, positive maps can also be defined by having the property of preserving the set of negative semidefinite matrices. Thus, it comes as a meaningful requirement to demand that quantitative measures of the lack of positive or negative semidefiniteness are monotonic with respect to this preorder together with the additional condition that they should vanish on the respective subsets. However, this preorder is trivial since it turns out that for any pair of non-definite matrices there exists positive maps $\Phi$ an $\Phi'$ such that $\Phi(A)=B$ and $\Phi'(B)=A$ and, therefore, $A\to B$ and $B\to A$ for any pair of Hermitian matrices $A,B$ which are neither positive semidefinite nor negative semidefinite. In order to fix this, an additional requirement must be added to the positivity of the map that will act as a gauge. A plausible condition and a natural ingredient in the theory of positive maps is precisely to demand that the positive map is either unital or trace-preserving. Notice that any of these two conditions forbids then positive maps that blow up or down the matrices: $\Phi(X)=cX$ for some real number $c>0$. Since both the set of PU maps and the set of PTP maps also contain the identity map and are closed under composition, both options lead as well to preorders in the set of Hermitian matrices. This motivates the following definitions.

\begin{definition}\label{monotones}
Let $\mu:\bigcup_{n\in\mathbb{N}}H_n\to\mathbb{R}$.
\begin{itemize}
\item $\mu$ is called a PU-monotone if $\mu(\Phi(A))\leq\mu(A)$ holds for every PU map $\Phi$ and every Hermitian matrix $A$.
\item $\mu$ is called a PTP-monotone if $\mu(\Phi(A))\leq\mu(A)$ holds for every PTP map $\Phi$ and every Hermitian matrix $A$.
\end{itemize}
\end{definition}

\begin{definition}\label{measures}
Let $\mu:\bigcup_{n\in\mathbb{N}}H_n\to[0,\infty)$.

\begin{itemize}
\item $\mu$ is called a PU-monotonic measure of the lack of positive semidefiniteness or PU$_{-}$ measure if it is a PU-monotone and $\mu(A)=0$ $\forall A\geq0$.
\item $\mu$ is called a PU-monotonic measure of the lack of negative semidefiniteness or PU$_{+}$ measure if it is a PU-monotone and $\mu(A)=0$ $\forall A\leq0$.
\item $\mu$ is called a PTP-monotonic measure of the lack of positive semidefiniteness or PTP$_{-}$ measure if it is a PTP-monotone and $\mu(A)=0$ $\forall A\geq0$.
\item $\mu$ is called a PTP-monotonic measure of the lack of negative semidefiniteness or PTP$_{+}$ measure if it is a PTP-monotone and $\mu(A)=0$ $\forall A\leq0$.
\end{itemize}
\end{definition}

Of course, one could demand that the positive map that defines the preorder was both unital and trace-preserving. However, there is no clear reason to ask for both conditions at the same time as we will see that one is enough in order to obtain a meaningful and non-trivial preorder. In fact, this enables not to impose that the source and target spaces of the maps are the same. Moreover, by making the conditions on the map less restrictive, one obtains more restrictive preorders and less functionals have the corresponding monotonicity property. Due to Theorem 1, any functional preserving the majorization relation for the eigenvalues of a Hermitian matrix (that is, any Schur-convex function of the eigenvalues \cite{Marshall}) will be monotonic under maps that are both PU and PTP. However, we will see that imposing one of these conditions in order to define the preorder leads to more specific functionals having the monotonicity property. Actually, it turns out that the measures of the lack of positive or negative semidefiniteness are essentially unique in both formalisms. Alternatively, in addition to providing a way of defining a unique measure for the lack of structure of a matrix, the theory of monotones borrowed from quantum resource theories gives technical means to characterize when linear transformations exist as each monotone leads to a necessary condition and the identification of all possible monotones to a sufficient condition.

Throughout this article I will use repeatedly that every Hermitian matrix $A$ can be decomposed uniquely into orthogonally supported positive and negative parts $A=A_+-A_-$ such that $A_+$ and $A_-$ are positive definite and $A_+A_-=0$. I will also follow the convention of Schatten $p$-norms and $||\cdot||_1$ will denote the trace norm (i.e.\ the sum of the singular values of the matrix) and $||\cdot||_\infty$ will denote the operator norm (i.e.\ the maximal singular value of the matrix). It is known that positive maps have contractivity properties under this norms. It follows from the Russo-Dye theorem (see e.g.\ \cite{Bhatia}) that $||\Phi(A)||_\infty\leq||A||_\infty$ for every PU map and for every Hermitian matrix $A$, i.e. $||\cdot||_\infty$ is a PU-monotone. It can also be proven (see e.g.\ \cite{Kossakowski1,Kossakowski2,Perez,Rivas}) that $||\Phi(A)||_1\leq||A||_1$ for every PTP map and for every Hermitian matrix $A$, i.e. $||\cdot||_1$ is a PTP-monotone\footnote{On the other hand, maps that are both PU and PTP are contractive under all Schatten $p$-norms for $1\leq p\leq \infty$ \cite{Perez}.}. Our main results are the following theorems, whose proofs are given in Secs. \ref{secn} and \ref{secs}.

\begin{theorem}\label{puth}
Let $A\in H_n$ and $B\in H_k$ such that they are neither positive nor negative semidefinite. Then, there exists a PU linear map $\Phi:H_n\to H_k$ such that $\Phi(A)=B$ if and only if $||A_+||_\infty\geq||B_+||_\infty$ and $||A_-||_\infty\geq||B_-||_\infty$.
\end{theorem}

\begin{theorem}\label{ptpth}
Let $A\in H_n$ and $B\in H_k$. Then, there exists a PTP linear map $\Phi:H_n\to H_k$ such that $\Phi(A)=B$ if and only if $\tr A=\tr B$, $||A_+||_1\geq||B_+||_1$ and $||A_-||_1\geq||B_-||_1$.
\end{theorem}

Notice that the very last condition in Theorem \ref{ptpth} is redundant since $\tr A=||A_+||_1-||A_-||_1$ and it is only provided for the sake of the exposition. It is worth pointing out that the condition in this theorem is equivalent to $\tr A=\tr B$ and $||A||_1\geq||B||_1$; thus, the aforementioned condition that $||\Phi(A)||_1\leq||A||_1$ for PTP maps is not only necessary but sufficient if supplemented with the obvious trace-preservation condition. This is in contrast with PU maps where Theorem \ref{puth} shows that the condition $||A||_\infty\geq||B||_\infty$ is necessary but not sufficient.

The conditions on the negative and positive parts of the matrices give us monotonic measures for the lack of positive and negative semidefiniteness respectively in both formalisms. Thus, Theorems \ref{puth} and \ref{ptpth} state that essentially there are unique PU$_{\pm}$ measures: $\mu_\infty^\pm(A)=||A_{\pm}||_\infty$ and PTP$_{\pm}$ measures: $\mu_1^\pm(A)=||A_{\pm}||_1$ (cf.\ Lemmas \ref{lemmapumon} and \ref{lemmaptpmon} below). The additional premise in Theorem \ref{puth} that the matrices are neither positive nor negative semidefinite is necessary in order to characterize the possible transformations in terms of PU$_{\pm}$ measures. If one is not interested in this relation, the condition can be reformulated in such a way that it includes the case of definite matrices.

\begin{theorem}\label{puthg}
Let $A\in H_n$ and $B\in H_k$. Then, there exists a PU linear map $\Phi:H_n\to H_k$ such that $\Phi(A)=B$ if and only if $\lambda_{\max}(A)\geq\lambda_{\max}(B)$ and $\lambda_{\min}(A)\leq\lambda_{\min}(B)$, where $\lambda_{\max}$ and $\lambda_{\min}$ denote respectively the maximal and minimal eigenvalues of a matrix.
\end{theorem}

This theorem is proven in Sec.\ \ref{secputhg} below, which in addition provides an alternative proof to Theorem \ref{puth} without making any explicit connection to monotonic measures of the lack of positive or negative semidefiniteness of a matrix.

The rest of this article is devoted to the technical work necessary to establish the above theorems. Before that, I would like to point out some concluding remarks. This paper characterizes the preorders induced in the set of Hermitian matrices by the action of linear positive and unital or trace-preserving maps. One of the main motivations was the derivation of measures of the lack of positive (or negative) semidefiniteness with meaningful monotonicity properties in the spirit of quantum resource theories in quantum information. One finds that such measures are essentially unique in both preorders, boiling down respectively to the 1-norm and the $\infty$-norm of the positive (or negative) part of the matrix. It turns out that these measures correspond to minimizing the distance to set of positive or negative semidefinite matrices in the corresponding norms, which are therefore singled out with respect to other choices of norm. However, it should be stressed that the axiomatic derivation taken here based on monotonicity does not necessarily lead to distance measures when extrapolated to other contexts. The analysis of the properties of structured matrices is ubiquitous in matrix theory. Many problems can be particularly well tackled provided one is dealing with matrices with a given structure while deviations from it can lead to quite a different behaviour. The construction of axiomatically-justified and well-behaved measures of the lack of structure of a matrix in the sense advocated here might provide a first step for the derivation of quantitative bounds on the differences in the behaviour one is interested in. A particular example of this is the analysis of eigenvalue algorithms for normal and non-normal matrices, which led Henrici \cite{Henrici} to introduce the so-called “departure from normality” and has motivated the study of measures of non-normality (see e.g. \cite{Elsner}). I have considered here the particular case in which the relevant structure is that of positive semidefinite matrices. However, the idea on which this work is based can be used for arbitrary structures. To explore these extensions and its usefulness in different applications of matrix theory is left as problem for future research. This would require the analysis of linear maps that preserve a given set of structured matrices. This is precisely the subject of study of the theory of linear preserver problems \cite{Li1,Li2,Pierce}, which can serve as a basis to carry out this program (for instance, linear normality-preserving maps are characterized in \cite{Kunicki}). On the other hand, looking at this connection the other way around, it is my hope that the constructions of monotonic measures from quantum resource theories utilized here to characterize transformations by positive maps might be of use in other linear preserver problems and related questions.

\section{Proof of Theorems \ref{puth} and \ref{ptpth}: Necessity of the conditions}\label{secn}

In this section I show that if there exists a PU or PTP map $\Phi$ such that $\Phi(A)=B$, then the conditions given in Theorems \ref{puth} and \ref{ptpth} on the matrices $A,B$ must hold. In order to do so, it suffices to see that $||A_{\pm}||_\infty$ are PU-monotones and $||A_{\pm}||_1$ are PTP-monotones (the trace preservation condition in the latter case is evident and I do not explicitly write it in this section). This task can be easily achieved using standard constructions of monotones from quantum resources theories (although direct proofs that do not rely on monotones can also be devised, see Sec.\ 5 below for the case of PU maps).
\begin{lemma}\label{robustness}
Let $A\in H_n$ and define
\begin{align}
\mu_\infty^-(A)&=\min\{p: A+pX\geq0, p\geq0, X\in H_n, ||X||_\infty\leq1\},\\
\mu_1^-(A)&=\min\{p: A+pX\geq0, p\geq0, X\in H_n, ||X||_1\leq1\}.
\end{align}
Then, $\mu_\infty^-$ is a PU$_-$ measure and $\mu_1^-$ is a PTP$_-$ measure.
\end{lemma}
\begin{proof}
Obviously, $\mu_\infty^-(A)=\mu_1^-(A)=0$ $\forall A\geq0$, so it remains to see that these quantities are monotones. Consider first $\mu_\infty^-$ and any PU map $\Phi$ acting on any Hermitian matrix $A$. By definition of the measure, there exists a Hermitian matrix $X$ such that $||X||_\infty\leq1$ and $A+\mu_\infty^-(A)X\geq0$. Then, the linearity and the positivity of $\Phi$ implies that $\Phi(A)+\mu_\infty^-(A)\Phi(X)\geq0$. Now, the fact that the map is unital tells us that $||\Phi(X)||_\infty\leq1$ and, then, together with the previous condition, this immediately implies that $\mu_\infty^-(\Phi(A))\leq\mu_\infty^-(A)$ for every PU map $\Phi$, i.e.\ $\mu_\infty^-$ is a PU-monotone. The argument to see that $\mu_1^-$ is a PTP monotone follows the same lines using that PTP maps are contractive with respect to the 1-norm.
\end{proof}
\begin{lemma}
Let $A\in H_n$. It holds that
\begin{align}
\mu_\infty^-(A)&=||A_-||_\infty,\\
\mu_1^-(A)&=||A_-||_1.
\end{align}
\end{lemma}
\begin{proof}
Clearly,
\begin{equation}
A+||A_-||_\infty\frac{A_-}{||A_-||_\infty}\geq0,\quad A+||A_-||_1\frac{A_-}{||A_-||_1}\geq0,
\end{equation}
which shows that $\mu_\infty^-(A)\leq||A_-||_\infty$ and $\mu_1^-(A)\leq||A_-||_1$. To see the inequality in the other direction, notice that $A+pX\geq0$ implies that
\begin{equation}\label{eqlemma2n}
p\lambda_i(X)\geq\lambda_i(-A)
\end{equation}
holds for the non-increasingly ordered eigenvalues of the matrices $\forall i$. Thus, considering the case $i=1$ and using that $\lambda_1(X)\leq||X||_\infty\leq1$ we have that
\begin{equation}
p\geq\frac{||A_-||_\infty}{\lambda_1(X)}\geq||A_-||_\infty,
\end{equation}
proving that $\mu_\infty^-(A)\geq||A_-||_\infty$. To obtain the analogous claim for $\mu_1^-$, let $k$ denote the number of negative eigenvalues of $A$. Then, $\sum_{i=1}^k\lambda_i(-A)=||A_-||_1$ and $\sum_{i=1}^k\lambda_i(X)\leq\tr X\leq||X||_1\leq1$. Thus, using again Eq.\ (\ref{eqlemma2n}) we obtain that
\begin{equation}
p\geq\frac{||A_-||_1}{\sum_{i=1}^k\lambda_i(X)}\geq||A_-||_1,
\end{equation}
and, therefore, $\mu_1^-(A)\geq||A_-||_1$.
\end{proof}
\begin{lemma}
Let $A\in H_n$ and define
\begin{align}
\mu_\infty^+(A)&=\mu_\infty^-(-A)=||A_+||_\infty,\\
\mu_1^+(A)&=\mu_1^-(-A)=||A_+||_1.
\end{align}
Then, $\mu_\infty^+$ is a PU$_+$ measure and $\mu_1^+$ is a PTP$_+$ measure.
\end{lemma}
\begin{proof}
As in the previous case it is clear that $\mu_\infty^+(A)=\mu_1^+(A)=0$ $\forall A\leq0$. Thus, it only remains to check that the corresponding monotonicity property holds; however, this follows straightforwardly from that of $\mu_\infty^-$ and $\mu_1^-$, i.e.\
\begin{equation}
\mu_{\infty,1}^+(\Phi(A))=\mu_{\infty,1}^-(-\Phi(A))=\mu_{\infty,1}^-(\Phi(-A))\leq\mu_{\infty,1}^-(-A)=\mu_{\infty,1}^+(A)
\end{equation}
for, depending on the corresponding case, every PU or PTP map $\Phi$.
\end{proof}

This concludes what needed to be proven in this section.
\begin{corollary}
Let $A\in H_n$ and $B\in H_k$. If there exists a PU (PTP) linear map $\Phi:H_n\to H_k$ such that $\Phi(A)=B$, then it must hold that $||A_+||_\infty\geq||B_+||_\infty$ and $||A_-||_\infty\geq||B_-||_\infty$ ($||A_+||_1\geq||B_+||_1$ and $||A_-||_1\geq||B_-||_1$).
\end{corollary}

\begin{remark}
The construction used in Lemma \ref{robustness} for $\mu_\infty^-$ and $\mu_1^-$ corresponds to a standard quantifier in quantum resource theories: the so-called global robustness \cite{Chitambar}. Many other choices are possible (cf.\ \cite{Chitambar}) for any $A\in H_n$ such as the norm distance
\begin{equation}
\inf_{X\geq0}||A-X||
\end{equation}
or the absolute robustness
\begin{equation}
\min\{p: A+pX\geq0, p\geq0, X\geq0, ||X||\leq1\},
\end{equation}
which can be readily verified to be PU or PTP monotones if the norms are respectively taken to be the $\infty$-norm or the $1$-norm using the contractivity under them of the corresponding maps. The reader can check that in this case these quantities also boil down to $||A_-||_{\infty}$ or $||A_-||_1$. This comes as no surprise in the light of the results of the next section. All PU$_-$ measures and all PU$_+$ measures are essentially bound to this form.
\end{remark}

\section{Proof of Theorems \ref{puth} and \ref{ptpth}: Sufficiency of the conditions}\label{secs}

In this section I show that if the conditions given in Theorems \ref{puth} and \ref{ptpth} on the matrices $A,B$ hold, then there exists a PU or PTP map $\Phi$ such that $\Phi(A)=B$. The line of argumentation is the same for both unital and trace-preserving maps. I first explicitly construct transformations that enable to identify matrices $A,B$ for which it holds that $A\to B$ and $B\to A$ under the preorders induced by both classes of positive maps. Then, I use this to characterize the form of any monotone. This, in turn, allows one to prove using a very simple reasoning the sufficiency of the aforementioned conditions in a non-constructive way.

\subsection{Unital maps}

\begin{lemma}\label{lemmaspu}
Let $A\in H_n$ and $B\in H_k$ be neither positive nor negative semidefinite matrices such that $||A_+||_\infty=||B_+||_\infty$ and $||A_-||_\infty=||B_-||_\infty$. Then, there exists a PU linear map $\Phi:H_n\to H_k$ such that $\Phi(A)=B$ and a PU linear map $\Phi':H_k\to H_n$ such that $\Phi'(B)=A$.
\end{lemma}
\begin{proof}
Clearly, it suffices to construct a linear PU map that maps any given non-definite $A\in H_n$ to any given non-definite $B\in H_k$ such that $||A_+||_\infty=||B_+||_\infty$ and $||A_-||_\infty=||B_-||_\infty$. Let $A$ and $B$ be such matrices with spectral decomposition $A=\sum_i\lambda_iP_i$ and $B=\sum_i\mu_iQ_i$ where $\{\lambda_i\}$ ($\{\mu_i\}$) are the eigenvalues of $A$ ($B$) and the $\{P_i\}$ and $\{Q_i\}$ are rank-one mutually orthogonal orthogonal projections such that $\sum_iP_i=\one$ and $\sum_iQ_i=\one$. The map $\Phi$ that I am going to construct has the property that $\Phi(P_i)=\sum_jp_{ij}Q_j$ where $p_{ij}\geq0$ $\forall i,j$. Such a map can always be defined with the property that it is positive by putting all elements in the orthogonal complement of span$\{P_i\}$ in $H_n$ (with respect to the standard scalar product $\langle X,Y\rangle=\tr(XY)$) in the kernel of $\Phi$. For the sake of explicitly describing the action of this map on the $\{P_i\}$ and to verify in addition its unitality, it is convenient to rewrite the spectral decompositions as
\begin{equation}
A=\sum_{i\in I_A^+}\lambda_i^+P_i-\sum_{i\in I_A^-}\lambda_i^-P'_i,\quad B=\sum_{i\in I_B^+}\mu_i^+Q_i-\sum_{i\in I_B^-}\mu_i^-Q'_i,
\end{equation}
where all index sets take values in $\mathbb{N}$ starting from 1, $\lambda_1^\pm=\mu_1^\pm=||A_\pm||_\infty=||B_\pm||_\infty$ and $\lambda_i^+,\mu_i^+\geq0$ $\forall i$ and $\lambda_i^-,\mu_i^->0$ $\forall i$, i.e. any possible zero eigenvalues are included in the sets $\{\lambda_i^+\}$ and $\{\mu_i^+\}$ so that
\begin{equation}
\sum_{i\in I_A^+}P_i+\sum_{i\in I_A^-}P'_i=\one,\quad \sum_{i\in I_B^+}Q_i+\sum_{i\in I_B^-}Q'_i=\one.
\end{equation}
Consider now the following subsets of indices
\begin{align}
I_1^\pm&=\{i:\lambda_i^\pm\geq\mu_i^\pm,i\neq1\}\nonumber\\
I_2^\pm&=\{i:\lambda_i^\pm<\mu_i^\pm\}\nonumber\\
I_3^\pm&=\{i:i\in I_A^\pm,i\notin I_B^\pm\}\nonumber\\
I_4^\pm&=\{i:i\notin I_A^\pm,i\in I_B^\pm\}.
\end{align}
The map (which can be easily verified to be unital) is then defined by
\begin{align}
\Phi(P_1)&=Q_1+\frac{\lambda_1^-}{\lambda_1^-+\lambda_1^+}\left[\sum_{i\in I_1^+}\left(1-\frac{\mu_i^+}{\lambda_i^+}\right)Q_i+\sum_{i\in I_1^-}\left(1-\frac{\mu_i^-}{\lambda_i^-}\right)Q'_i\right]\nonumber\\
&+\sum_{i\in I_2^+}\left(1-\frac{\lambda_1^+-\mu_i^+}{\lambda_1^+-\lambda_i^+}\right)Q_i+\sum_{i\in I_4^+}\frac{\lambda_1^-+\mu_i^+}{\lambda_1^-+\lambda_1^+}Q_i+\sum_{i\in I_4^-}\frac{\lambda_1^--\mu_i^-}{\lambda_1^-+\lambda_1^+}Q'_i\nonumber\\
\Phi(P'_1)&=Q'_1+\frac{\lambda_1^+}{\lambda_1^-+\lambda_1^+}\left[\sum_{i\in I_1^+}\left(1-\frac{\mu_i^+}{\lambda_i^+}\right)Q_i+\sum_{i\in I_1^-}\left(1-\frac{\mu_i^-}{\lambda_i^-}\right)Q'_i\right]\nonumber\\
&+\sum_{i\in I_2^-}\left(1-\frac{\lambda_1^--\mu_i^-}{\lambda_1^--\lambda_i^-}\right)Q'_i+\sum_{i\in I_4^+}\frac{\lambda_1^+-\mu_i^+}{\lambda_1^-+\lambda_1^+}Q_i+\sum_{i\in I_4^-}\frac{\lambda_1^++\mu_i^-}{\lambda_1^-+\lambda_1^+}Q'_i\nonumber\\
\Phi(P_i)&=\frac{\mu_i^+}{\lambda_i^+}Q_i\quad (i\in I_1^+),\quad\quad\Phi(P'_i)=\frac{\mu_i^-}{\lambda_i^-}Q'_i\quad (i\in I_1^-)\nonumber\\
\Phi(P_i)&=\frac{\lambda_1^+-\mu_i^+}{\lambda_1^+-\lambda_i^+}Q_i\quad (i\in I_2^+),\quad\quad\Phi(P'_i)=\frac{\lambda_1^--\mu_i^-}{\lambda_1^--\lambda_i^-}Q'_i\quad (i\in I_2^-)\nonumber\\
\Phi(P_i)&=0\quad (i\in I_3^+),\quad\quad\Phi(P'_i)=0\quad (i\in I_3^-).
\end{align}
\end{proof}

\begin{corollary}\label{corollarypumon}
Any PU-monotone $\mu(A)$ is a function of $||A_+||_\infty$ and $||A_-||_\infty$ when restricted to non-definite matrices.
\end{corollary}
\begin{proof}
Suppose that the statement of the corollary was not true. Then, there would exist a pair of Hermitian non-definite matrices $A,B$ such that $||A_+||_\infty=||B_+||_\infty$ and $||A_-||_\infty=||B_-||_\infty$ and a PU-monotone $\mu$ such that $\mu(A)<\mu(B)$. However, this would be in contradiction with the fact that there exists a PU map $\Phi$ such that $\Phi(A)=B$ as Lemma \ref{lemmaspu} dictates.
\end{proof}

\begin{lemma}\label{lemmapumon}
Any PU-monotone $\mu(A)$ is a monotonically non-decreasing function of both $||A_+||_\infty$ and $||A_-||_\infty$ when restricted to non-definite matrices.
\end{lemma}
\begin{proof}
Using the previous corollary, we only need to see that for every pair of real numbers $r_1>r_2>0$ there exists a Hermitian matrix $A$ and a PU map $\Phi$ such that $||A_+||_\infty=r_1$, $||(\Phi(A))_+||_\infty=r_2$ for any value of $||A_-||_\infty=||(\Phi(A))_-||_\infty>0$ (notice that by linearity we then have that $\Phi(-A)=-\Phi(A)$ and this automatically implies that the function must also be monotonically non-decreasing with respect to the norm of the negative part). For this it suffices to construct a PU map $\Phi:H_2\to H_2$ such that
\begin{equation}
\Phi\left[\left(
            \begin{array}{cc}
              1 & 0 \\
              0 & -x \\
            \end{array}
          \right)\right]=\left(
                           \begin{array}{cc}
                             p & 0 \\
                             0 & -x \\
                           \end{array}
                         \right)
\end{equation}
for any $0<p<1$ and any $x>0$. The following map (which is obviously PU for any $0\leq q\leq1$),
\begin{equation}
\Phi\left[\left(
            \begin{array}{cc}
              a & c \\
              \bar c & b \\
            \end{array}
          \right)\right]=\left(
                           \begin{array}{cc}
                             qa + (1-q)b & 0 \\
                             0 & b \\
                           \end{array}
                         \right),
\end{equation}
does the job choosing
\begin{equation}
q=\frac{p+x}{1+x}\in(0,1).
\end{equation}
\end{proof}
This allows us to obtain the desired conclusion.
\begin{lemma}\label{lemmaspufin}
Let $A\in H_n$ and $B\in H_k$ be neither positive nor negative definite matrices such that $||A_+||_\infty\geq||B_+||_\infty$ and $||A_-||_\infty\geq||B_-||_\infty$. Then, there exists a PU linear map $\Phi:H_n\to H_k$ such that $\Phi(A)=B$.
\end{lemma}
\begin{proof}
It is a consequence of the premise and Lemma \ref{lemmapumon} that for every PU-monotone $\mu$ it must hold that $\mu(A)\geq \mu(B)$. Considering now the following functional, which is obviously by construction a PU-monotone:
\begin{equation}
\mu_B(X)=\left\{\begin{array}{c}
                  0 \quad\nexists\Phi\in PU :\Phi(X)=B, \\
                  1 \quad\exists\Phi\in PU :\Phi(X)=B,
                \end{array}\right.
\end{equation}
it must then hold that $\mu_B(A)\geq\mu_B(B)=1$. Therefore, $\mu_B(A)=1$ and the result follows.
\end{proof}

\subsection{Trace-preserving maps}

\begin{lemma}\label{lemmasptp}
Let $A\in H_n$ and $B\in H_k$ be such that $||A_+||_1=||B_+||_1$ and $||A_-||_1=||B_-||_1$. Then, there exists a PTP linear map $\Phi:H_n\to H_k$ such that $\Phi(A)=B$ and a PTP linear map $\Phi':H_k\to H_n$ such that $\Phi'(B)=A$.
\end{lemma}
\begin{proof}
As in the PU case, it suffices to construct a linear PTP map that takes any given $A\in H_n$ to any given $B\in H_k$ such that $||A_+||_1=||B_+||_1$ and $||A_-||_1=||B_-||_1$. Suppose that $A$ and $B$ are unitarily diagonalized as $A=UD_AU^\ast$ and $B=VD_BV^\ast$, where the diagonal matrices are block-partitioned as
\begin{equation}
D_A=\left(
      \begin{array}{cc}
        D_A^+ & 0 \\
        0 & -D_A^- \\
      \end{array}
    \right),\quad D_B=\left(
      \begin{array}{cc}
        D_B^+ & 0 \\
        0 & -D_B^- \\
      \end{array}
    \right),
\end{equation}
where $D_A^+\in H_{r},D_B^+\in H_s$ are diagonal positive semidefinite matrices and $D_A^-\in H_{n-r},D_B^-\in H_{k-s}$ are diagonal positive definite matrices with the property that $\tr D_A^+=\tr D_B^+$ and $\tr D_A^-=\tr D_B^-$. I will now construct explicitly a map $\Phi:H_n\to H_k$ achieving $\Phi(A)=B$ as a composition of maps that are obviously PTP with the result that $\Phi$ must be PTP then too. In particular, let $\Phi=\Phi_3\circ\Phi_2\circ\Phi_1$ with $\Phi_1:H_n\to H_n$, $\Phi_2:H_n\to H_k$ and $\Phi_3:H_k\to H_k$ given by
\begin{align}
&\Phi_1(X)=U^\ast X U,\quad \Phi_3(X)=VXV^\ast,\nonumber\\
&\Phi_2\left[\left(
              \begin{array}{cc}
                X & Z \\
                Z^\ast & Y \\
              \end{array}
            \right)\right]=\left(
                             \begin{array}{cc}
                               \Psi_+(X) & 0 \\
                               0 & \Psi_-(Y) \\
                             \end{array}
                           \right),
\end{align}
where $\Psi_+:H_r\to H_s$ and $\Psi_-:H_{n-r}\to H_{k-s}$ are the PTP maps given by
\begin{equation}
\Psi_\pm(X)=\frac{\tr X}{\tr D_B^\pm}D_B^\pm.
\end{equation}
\end{proof}

\begin{corollary}\label{corollaryptpmon}
Any PTP-monotone $\mu(A)$ is a function of $||A_+||_1$ and $||A_-||_1$.
\end{corollary}
\begin{proof}
Same argument as in Corollary \ref{corollarypumon} using now Lemma \ref{lemmasptp}.
\end{proof}

\begin{lemma}\label{lemmaptpmon}
Any PTP-monotone $\mu(A)$ is a function of $\tr A$ and $||A_+||_1$, being in addition monotonically non-decreasing in this second argument.
\end{lemma}
\begin{proof}
That any monotone is a function of $\tr A$ and $||A_+||_1$ follows straightforwardly from the previous corollary by noticing that $||A_-||_1=||A_+||_1-\tr A$. Thus, as in the unital case, it remains to see that for any $t\in\mathbb{R}$ and for any pair of real numbers $r_1>r_2\geq \max\{t,0\}$ there exists a Hermitian matrix $A$ and a PTP map $\Phi$ such that $||A_+||_1=r_1$, $||(\Phi(A))_+||_1=r_2$ and $\tr A=\tr\Phi(A)=t$. For this it suffices to construct a PTP map $\Phi:H_2\to H_2$ such that
\begin{equation}
\Phi\left[\left(
            \begin{array}{cc}
              x & 0 \\
              0 & -y \\
            \end{array}
          \right)\right]=\left(
                           \begin{array}{cc}
                             px & 0 \\
                             0 & x-y-px \\
                           \end{array}
                         \right)
\end{equation}
for any $x,y>0$ and any $\max\{0,(x-y)/x\}\leq p<1$. The following map (which is obviously PTP for any $0\leq p\leq1$) does the job:
\begin{equation}
\Phi\left[\left(
            \begin{array}{cc}
              a & c \\
              \bar c & b \\
            \end{array}
          \right)\right]=\left(
                           \begin{array}{cc}
                             pa & 0 \\
                             0 & b+(1-p)a \\
                           \end{array}
                         \right).
\end{equation}
\end{proof}

As in the previous subsection, we are now in the position to obtain the desired claim.

\begin{lemma}
Let $A\in H_n$ and $B\in H_k$ be such that $\tr A=\tr B$ and $||A_+||_1\geq||B_+||_1$ (and, hence, $||A_-||_1\geq||B_-||_1$). Then, there exists a PTP linear map $\Phi:H_n\to H_k$ such that $\Phi(A)=B$.
\end{lemma}

\begin{proof}
The argument is completely analogous to that of Lemma \ref{lemmaspufin} for PU maps. By Lemma \ref{lemmaptpmon}, the conditions on the matrices $A,B$ impose that $\mu(A)\geq \mu(B)$ must hold for every PTP-monotone $\mu$. Considering now the particular PTP-monotone:
\begin{equation}
\mu_B(X)=\left\{\begin{array}{c}
                  0 \quad\nexists\Phi\in PTP :\Phi(X)=B, \\
                  1 \quad\exists\Phi\in PTP :\Phi(X)=B,
                \end{array}\right.
\end{equation}
we obtain that $\mu_B(A)\geq\mu_B(B)=1$, which leads to the desired result.
\end{proof}

\section{Proof of Theorem \ref{puthg}}\label{secputhg}

In this section I prove Theorem \ref{puthg}. This serves as an alternative proof of Theorem \ref{puth} without relying on the theory of monotones and allows to characterize transformations by PU maps without making an explicit connection to the quantification of the lack of positive semidefiniteness of a matrix.

%\begin{theorem}
%Let $A\in H_n$ and $B\in H_k$. Then, there exists a PU linear map $\Phi:H_n\to H_k$ such that $\Phi(A)=B$ if and only if %$\lambda_{\max}(A)\geq\lambda_{\max}(B)$ and $\lambda_{\min}(A)\leq\lambda_{\min}(B)$, where $\lambda_{\max}$ and $\lambda_{\min}$ denote %respectively the maximal and minimal eigenvalues of a matrix.
%\end{theorem}
\begin{proof}
On the one hand, in order to see the implication in one direction, notice that for any $A\in H_n$ it holds that $A-\lambda_{\min}(A)\one\geq0$ and $-A+\lambda_{\max}(A)\one\geq0$. Thus, for any PU map $\Phi:H_n\to H_k$ we have that
\begin{align}
\Phi(A-\lambda_{\min}(A)\one)=\Phi(A)-\lambda_{\min}(A)\one\geq0&\Rightarrow\lambda_{\min}(\Phi(A))\geq\lambda_{\min}(A),\nonumber\\
\Phi(-A+\lambda_{\max}(A)\one)=-\Phi(A)+\lambda_{\max}(A)\one\geq0&\Rightarrow\lambda_{\max}(\Phi(A))\leq\lambda_{\max}(A).
\end{align}
On the other hand, in order to see the implication in the other direction, let $A\in H_n$ and $B\in H_k$ fulfill the eigenvalue conditions of the statement of the theorem and denote their spectral decompositions by
\begin{equation}
A=\sum_i\lambda_iP_i,\quad B=\sum_i\mu_iQ_i,
\end{equation}
where $\{\lambda_i\}$ ($\{\mu_i\}$) are the eigenvalues of $A$ ($B$) and the $\{P_i\}$ and $\{Q_i\}$ are rank-one mutually orthogonal orthogonal projections such that $\sum_iP_i=\one$ and $\sum_iQ_i=\one$. The fact that $\lambda_{\max}(A)\geq\lambda_{\max}(B)$ and $\lambda_{\min}(A)\leq\lambda_{\min}(B)$ means that every eigenvalue of $B$ lies in the convex hull of the eigenvalues of $A$ and, therefore, for every $j$ it holds that $\mu_j=\sum_ip_{ij}\lambda_i$ with $p_{ij}\geq0$ $\forall i$ and $\sum_ip_{ij}=1$. Using these convex weights, we can define now a map $\Phi$ to act as $\Phi(P_i)=\sum_jp_{ij}Q_j$ $\forall i$ and having in its kernel all elements in the orthogonal complement of span$\{P_i\}$ in $H_n$. This map is manifestly positive and
\begin{align}
\Phi(\one)&=\Phi\left(\sum_iP_i\right)=\sum_{ij}p_{ij}Q_j=\sum_jQ_j=\one,\nonumber\\
\Phi(A)&=\Phi\left(\sum_i\lambda_iP_i\right)=\sum_{ij}p_{ij}\lambda_iQ_j=\sum_j\mu_jQ_j=B.
\end{align}
\end{proof}

\section*{Acknowledgements}

I am grateful to the following institutions for financial support: Spanish MINECO (grants
MTM2017-88385-P and MTM2017-84098-P) and Comunidad de Madrid (grant QUITEMAD-CMS2018/TCS-4342 and the Multiannual Agreement with UC3M in the line of Excellence of University Professors EPUC3M23 in the context of the V PRICIT).

\end{document}